\documentclass[12pt]{article}

\usepackage{hyperref}
\hypersetup
{
	colorlinks,
	citecolor=blue,
	filecolor=green,
	linkcolor=magenta,
	urlcolor=blue
}
\hypersetup{linktocpage}

\usepackage{fancyhdr}
\usepackage{amsfonts}
\usepackage{amsmath}
\usepackage{amssymb}
\usepackage{amsthm}
\usepackage{tikz}
\usepackage{accents}

\usepackage{mathrsfs}
\usepackage{amscd}
\usepackage{amstext}
\usepackage{authblk}

\usepackage{graphics}
\usepackage{graphicx}

\newlength{\fixboxwidth}
\newcommand{\re}{\mathbb{R}}\newcommand{\N}{\mathbb{N}}
\newcommand{\zz}{\mathbb{Z}}

\newcommand{\com}{\mathbb{C}}
\newcommand{\Z}{{\zz}^d}

\newcommand{\R}{{\re}^d}

\newcommand\dx{\,{\mathrm d}}
\newcommand{\cs}{{\mathcal S}}

\newcommand{\cf}{{\mathcal F}}
\newcommand{\cfi}{{\cf}^{-1}}

\newcommand{\supp}{{\rm supp \, }}

\newcommand{\gf}{\mathcal{F}}
\newcommand{\unif}{{\rm unif}}

\newcommand{\bproof}{\begin{proof}}
	\newcommand{\eproof}{\end{proof}}

\newcommand{\be}{\begin{equation}}
	\newcommand{\ee}{\end{equation}}
\newcommand{\beq}{\begin{eqnarray}}
	\newcommand{\beqq}{\begin{eqnarray*}}
		\newcommand{\eeq}{\end{eqnarray}}
	\newcommand{\eeqq}{\end{eqnarray*}}

\numberwithin{equation}{section}

\newtheorem{theorem}{Theorem}[section]
\newtheorem{definition}[theorem]{Definition}

\newtheorem{lemma}[theorem]{Lemma}
\newtheorem{proposition}[theorem]{Proposition}
\newtheorem{remark}[theorem]{Remark}

\begin{document}

	\title{Triebel--Lizorkin spaces of dominating mixed smoothness are algebras with respect to pointwise multiplication}
	
	\author{Van Kien Nguyen}
	\affil{University of Transport and Communications
		\\
		No 3 Cau Giay Street, Lang Ward, Hanoi, Vietnam
\\	Email: kiennv@utc.edu.vn}

	
	\date{\today}
	
	\maketitle
	\begin{abstract}
We continue our investigations on pointwise multipliers for spaces of dominating mixed
smoothness.
In this paper we prove that under a certain conditions Triebel--Lizorkin spaces of dominating mixed smoothness form algebras with respect to pointwise multiplication.
	\end{abstract}

	

\section{Introduction}

The regularity concept related to  spaces of dominating mixed smoothness 
are standard in {Approximation Theory} \cite{Tem93B}, {Numerical Analysis} \cite{BuGr04,SST08}  and 
{Information-Based Complexity} \cite{NoWo08,NoWo10,NoWo12}.
These spaces appear naturally in many models of real world
problem in mathematical physics, finance and other fields. For instance, in a recent work
on regularity properties of solutions of the electronic Schr\"odinger equation, Yserentant \cite{Yser10}
has shown that the eigenfunctions of the electronic Schr\"odinger operator have a certain
mixed smoothness. Triebel \cite[Chapter 6]{Tri15B} has indicated a relation between Faber bases and
sampling recovery in the context of spaces with mixed smoothness and solutions of Navier-Stokes equations. In particular, when initial data belongs to spaces with mixed smoothness,
Navier-Stokes equations admit a unique solution.

Assertions on pointwise multipliers are among the key problems in the modern theory of function spaces and have found numerous applications in fields such as partial differential equations and numerical analysis \cite{BH21,MS09B}.
Pointwise multipliers for isotropic Besov and Triebel-Lizorkin spaces, over a wide range of parameters, have been studied in detail in \cite{MS85B,MS09B,Netrusov95,Tri83B,RS96B,Tri06B,LSYY24}.
In our previous papers \cite{NS1702,NS1701} we investigated 
the set of all pointwise multipliers $M(S^r_{p}W(\R))$ and $M(S^r_{p,q}B(\R))$ for the Sobolev space $S^r_{p}W(\R)$ and Besov spaces $S^r_{p,q}B(\R)$ of dominating mixed smoothness, respectively. 
It turned out that, under the natural restrictions $1\leq p<\infty$, $r\in \N$ for $S^r_{p}W(\R)$  and $1\le p\leq q \le \infty$, $r>1/p$ for $S^r_{p,q}B(\R)$, these sets are given by $S^r_{p}W(\R)_\unif$ and $S^r_{p,q}B(\R)_\unif$.
Formally, the assertion is completely parallel to the isotropic case.
However, the extension from the isotropic case to the dominating mixed case is by no means straightforward.
The standard method used in the isotropic setting, namely paramultiplication, does not appear to be appropriate.
In fact, the proof of the result in the dominating mixed case is much more involved.
Our method in these papers is based on the characterization of the underlying spaces by differences, sometimes combined with their Fourier-analytic description.

In the present paper our aim is to prove that the Triebel--Lizorkin spaces of dominating mixed smoothness form an algebra with respect to pointwise multiplication under the conditions  $0<p \le q \le \infty$, $p<\infty$ and $r>1/p$. As a consequence, we obtain that $S^r_{p,q}F(\R)\subset M(S^r_{p,q}F(\R)) $. Our method in this paper, as in \cite{NS1701,NS1702}, is the characterization  of the underlying spaces by differences but the proof is much more complicated than that in \cite{NS1701,NS1702} since the Nilkol'skijs inequality can not be applied in this situation, see \cite[page 2250]{NS1702}. At the present we do not know whether this method also works in the case $0<q<p<\infty$. However, by interpolation, we can show that the Bessel-potential space of dominating mixed smoothness $S^r_{p}H(\R)$ is a multiplication algebra for $1< p <\infty$ and  $r>\max(1/p,1-1/p)$.

The paper is organized as follows.
In Section \ref{def} we define the classes $S^r_{p,q}F(\R)$ by using  the Fourier analytic approach and give some conditions under which the spaces  $S^r_{p,q}F(\R)$ can be characterized by mixed differences. In Section \ref{sec-tools}  we will collect some required tools    from Fourier analysis 
and few basic inequalities for differences. 
Finally, in Section \ref{proofs},  we state our main results and present detailed proofs. 



\noindent
{\bf Notation.}  As usual $\N$ denotes the natural numbers, $\N_0 :=\N\cup\{0\}$, 
$\zz$ denotes the integers, 
$\re$ the real numbers, 
and $\com$ the complex numbers. The letter $d\in \N, \ d>1,$ is always reserved for the underlying dimension in $\R, \Z$ etc. For $x \in \R$, $x_i$ denotes the $i$th coordinate, i.e., $x := (x_1,\ldots, x_d)$. If $ 0< p\leq \infty$, we write
$|x|_p := \big(\sum_{i=1}^d |x_i|^p\big)^{1/p}$ with the usual modification when $p=\infty$. 
Further,
by  $\langle x, y\rangle$  we mean the usual Euclidean inner product in $\R$. The notation
$x\diamond y$ stands for $ (x_1y_1, \ldots  , x_dy_d)\in \R\,.$ The norm
of an element $x$ in the normed space $X$ will be denoted by $\|x\,|\,X\|$.


\section{Triebel--Lizorkin spaces of dominating mixed smoothness}\label{def}



We introduce the Triebel--Lizorkin spaces of dominating mixed smoothness by using the Fourier analytic approach. Let $\cs(\R)$ be the Schwartz space of all complex-valued rapidly decreasing infinitely differentiable  functions on $\R$. 
The topological dual of $\cs(\R)$, the class of tempered distributions equipped with the weak topology, is denoted by $\cs'(\R)$.
The Fourier transform on $\cs(\R)$ is given by 
\[
\cf \varphi (\xi) = (2\pi)^{-d/2} \int\limits_{\R} \, e^{-{\rm i} \langle x, \xi\rangle }\, \varphi (x)\, \dx x \, , \qquad  \xi \in \R\, .
\]
The inverse transformation is denoted by $\cfi $.
We use both notations also for the transformations defined on $\cs'(\R)$\,.

Let $\varphi_0 \in C_0^{\infty}({\re})$ be a non-negative function such that 
 $\varphi_0  \equiv  1$ on $[-1,1]$ and $\supp(\varphi_0) \subset [-\frac{3}{2},\frac{3}{2}]$. 
For $j\in \N$ we define
\begin{equation*}
         \varphi_j(\xi) = \varphi_0(2^{-j}\xi)-\varphi_0(2^{-j+1}\xi) ,\qquad\ \xi \in \re\, , 
\end{equation*}
and  
\be\label{ws-101}
\varphi_{k}(x) := \varphi_{k_1}(x_1)\cdot \ldots \cdot
         \varphi_{k_d}(x_d)\, , \qquad  x \in \R, \quad k\in \N_0^d\,. 
\ee        
This implies 
\[
\sum_{k\in \N_0^d} \varphi_k(x) =  1 \qquad  \text{for all}\ x\in \R\, , 
\]
and
\[
\supp(\varphi_k)  \subset  \Big\{ x \in \R: 2^{k_\ell-1}\le |x_\ell| \le 3 \cdot 2^{k_\ell-1}\, , \quad \ell =1, \ldots   , d\Big\}\,, \qquad k \in \N^d\, .
\]
With other words, $(\varphi_k )_{k \in \N_0^d}$ is a smooth dyadic decomposition of unity of tensor product type in $\R$.

\begin{definition}\label{eq-def1}
Let $(\varphi_k)_{k\in \N_0^d}$ be the above system. Let $0<p<\infty$, $0 < q\leq \infty$ and $r \in \re$.  Then the Triebel--Lizorkin spaces of dominating mixed smoothness $ S^{r}_{p,q}F(\re^d)$ is the
         collection of all tempered distributions $f \in \mathcal{S}'(\R)$
         such that
\begin{equation*}
          \|\, f \, |S^r_{p,q}F(\R)\|_{\varphi} :=\Bigg\|
         \Bigg\{\sum\limits_{k\in \N_0^d} 2^{r|k|_1  q}\, |\, \cfi[\varphi_{k} \cf f]
         |^q\Bigg\}^{1/q} \Bigg|L_p(\R)\Bigg\|<\infty
\end{equation*} 
with the ususal modifications if $q=\infty$.
\end{definition}

\begin{remark}\label{useful}\rm 
	The scales $S^r_{p,q}F(\R)$ are quasi-Banach spaces. They are Banach spaces if $\min(p,q)\geq 1$. The  spaces $ S^r_{p,q}F(\R)$ are independent of the chosen decomposition of unity up to
	equivalence of the quasi-norms in the corresponding spaces, see e.g. \cite[Theorem 2.2.4]{ST87B}. Hence, we shall not indicate the index ``$\varphi$" in the sequel, i.e., we simply write $\|f|S^r_{p,q}F(\R)\|$ instead of $\|f|S^r_{p,q}F(\R)\|_{\varphi}$. 
\end{remark}

\begin{definition}\label{sobolev}\rm Let $1<p<\infty$.\\
	{\rm(i)} Let $m\in \N_0$. Then the  Sobolev space of dominating mixed 
	smoothness $S^m_p W(\R)$ is the collection of all  
	$f \in L_p (\R)$ such that all distributional derivatives $D^\alpha f$ with $|\alpha|_{\infty}\le m$
	belong to $L_p (\R)$. We put
	\beqq
	\|\, f\, |S^m_pW(\R)\| := \sum_{|\alpha|_{\infty} \leq m} \, \|\, D^{\alpha}f\, |L_p(\R)\|\,.
	\eeqq 
	{\rm (ii)} Let $r\in \re$. Then Bessel-potential space of dominating mixed smoothness $S^r_{p}H(\R)$ is the collection of all $f\in \cs'(\R)$ such that
	\beqq
	\| f|S^r_{p}H(\R)\|=\big\|\gf^{-1}\big[ (1+ y_1^2)^{r/2}\cdot\ldots\cdot (1+ y_d^2)^{r/2} \gf f \big](\cdot)\big|L_p(\R)\big\| 
	\eeqq
	is finite.  
\end{definition}
\begin{remark}\rm From the Definition \ref{sobolev} we have at once $S^0_pW(\R)= S^0_pH(\R)=L_p(\R)$. In general, if $m\in \N_0$ we have $S^m_pW(\R)= S^m_pH(\R)$ in the sense of equivalent norms, see \cite[Theorem 2.3.1]{ST87B}.
\end{remark}

\begin{remark}\rm
 The scales of functions $S^r_{p,q}F(\R)$ are generalizations of $S^r_pH(\R)$. This is  a consequence  of the Littlewood-Paley type assertion for spaces of dominating mixed
	smoothness which will be stated in the following.
\end{remark}
\begin{proposition}\label{littlewood}
	Let $r\in \re$ and $1 < p< \infty$. Then we have $ S_{p_,2}^{r} F (\R) = S_p^r H (\R) $
	in the sense of equivalent norms, i.e, it holds
	\begin{equation*}
	\|\, f \, |S^r_p H(\R)\|\asymp
	\Bigg\| \Bigg(\sum_{k \in \N_0^d} 2^{2r|k|_1}|\gf ^{-1}(\varphi_{k}\gf f)|^2\Bigg)^{1/2}\Bigg|L_p (\R)\Bigg\| 
	\end{equation*}
	for all $f \in S^r_pH (\R)$. For $r=0$ we get back the Littlewood-Paley assertion $L_p(\R)=S_{p_,2}^{0} F (\R)$.
\end{proposition}



We now turn to the characterization of the Triebel--Lizorkin spaces of dominating mixed smoothness by differences.
Let $j \in [d]= \{1,2, \ldots  , d\}$, $m \in \N$, $h \in \re$ and $x \in \R$. We define the $m$-th order difference of $f$ in direction $j$ by
\begin{equation*}
 \Delta_{h,j}^{m} f(x):= \sum_{\ell =0}^{m} (-1)^{m-\ell} \, \binom{m}{\ell} \, 
f(x_1, \ldots  , x_{j-1}, x_j + \ell h, x_{j+1}, \ldots  , x_d)\, .
\end{equation*}
For $e\subset [d]$, $ h  \in \R$ and $m \in \N_0^d$ the mixed $(m,e)$-th difference operator $\Delta_h^{m,e}$ is defined to be 
 \[
 \Delta_{ h}^{m,e} := \prod_{i \in e} \Delta_{h_i,i}^{m_i} \qquad\mbox{and}\qquad \Delta_h^{m,\emptyset} :=  \operatorname{Id} \,, 
 \]
where $\operatorname{Id}f = f$. 
For a set $e\subset [d]$ we denote
$$\N_0^d(e):=\big\{k\in \N_0^d: \ k_i=0\ \text{if}\ i\not \in e\big\}.$$
Let $k \in \N_0^d$.
For brevity we write $2^{-k} $ instead of the vector $(2^{-k_1}, 2^{-k_2}, \ldots  , 2^{-k_d})$. We have the following theorem whose proof can be found in \cite{Ull06} and  \cite[Section 2.3.3]{ST87B}  for the case $d=2$.

\begin{proposition}\label{diff} 
Let $0 < p <\infty$, $0< q\leq \infty$ and $r>\frac{1}{\min(p,q)}$. Let $m=(m_1,\ldots,m_d)\in \N^d$ such that $r < m_i$, $i=1,\ldots,d$. 
Then the Triebel--Lizorkin space of dominating mixed smoothness $S^r_{p,q}F(\R)$
is the collection of all $f\in L_p(\R)$ such that
$$
    \big\|\, f \, \big| S^r_{p,q}F(\R)\big\|_{(m)} := 
\sum_{e\subset [d]}\Bigg\|\Bigg\{\sum\limits_{k\in \N_0^d(e)} 2^{r|k|_1 q} \sup_{|h_i|\leq 2^{-k_i}\atop i\in e}\big|\Delta_{ h}^{m,e}f(\cdot)\big|^{q}\Bigg\}^{1/q}\Bigg|L_p(\R)\Bigg\|\, 
$$
is finite (with the usual modification if $q=\infty$). 
Furthermore,  $\|\, \cdot \, | S^r_{p,q}F(\R)\|_{(m)}$ generates a norm equivalent to $\|\,\cdot \, | S^r_{p,q}F(\R)\|$ on $L_p (\R)$.
\end{proposition}


\section{Tools from Fourier analysis}\label{sec-tools}


In this section we will collect some required tools from Fourier analysis.  
The following construction of a maximal function is essentially  due to Peetre, but based on earlier work of Fefferman and Stein.
Let $a>0$ and  $b=(b_1,\ldots,b_d)$, $b_i>0$, $i=1,\ldots,d$ be fixed. Let $f$ be a regular distribution  such that  $\gf f$ is compactly
supported. We define the Peetre maximal function $P_{b,a}f$ by
\begin{equation*}
P_{b,a}f(x) = \sup\limits_{z\in \R} \frac{|f(x-z)|}{\prod_{i=1}^d(1+|b_iz_i|)^a}
= \sup\limits_{y\in \R} \frac{|f(y)|}{\prod_{i=1}^d(1+|b_i(x_i-y_i)|)^a},\ x\in \R . 
\end{equation*}

\begin{lemma}\label{maxunglem} Let $0<p\leq \infty$ and $\Omega \subset \R$ be a compact set. Let further $a>0$
	and $\alpha = (\alpha_1,\ldots,\alpha_d) \in \N_0^d$. Then there exists a constant 
	$C>0$ (independent of $f$) such that
	\begin{equation*}
	\begin{split}
		P_{(1,\ldots ,1),a}(D^{\alpha}f)(x)\,  \leq\, C P_{(1,\ldots ,1),a} f(x)  
	\end{split}
	\label{g1}
	\end{equation*}
	holds for all $f\in L_p(\R)$ with $\supp(\gf f)\subset \Omega$ and $x\in \R.$  
\end{lemma}
\begin{remark}\rm The proof of Lemma \ref{maxunglem} can be found in \cite[Theorem 1.6.4]{ST87B}. Note that the constant $C$ there depends on $\Omega$. Let $\Omega=[-b_1,b_1]\times \ldots \times [-b_d,b_d]$, $b_i>0$, $i=1,\ldots,d$ and $\supp(\gf f)\subset \Omega$. Then applying Lemma \ref{maxunglem} for the function $f(x_1/b_1,\ldots,x_d/b_d)$ we obtain
	\begin{equation*}\label{eq-Peetre}
	\begin{split}
		P_{b,a}(D^{\alpha}f)(x) \,\leq \, C \Bigg(\prod_{i=1}^db_i^{\alpha_i}\Bigg)P_{b,a} f(x) 
	\end{split}
	\end{equation*}
	for all $f\in L_p(\R)$ with $\supp(\gf f)\subset \Omega$ and $x\in \R$. The constant $C$ now is independent of $f$ and $b=(b_1,\ldots,b_d)$.
\end{remark}
We will need the vector-valued Peetre maximal inequality. For more details see \cite[Theorem 1.10.2]{ST87B} and \cite[Proposition 2.3.4]{Han10}.

\begin{proposition}\label{peetremax}
Let $0<p< \infty$, $0<q\leq\infty$ and $\{\Omega_{k}\}_{k\in \N_0^d} $ be a sequence of compact subsets of $\R$ 
\begin{equation*}\label{comp}
\Omega_{k}=\big\{x\in \R: \ |x_{k_i}|\leq b_{k_i},\ i=1,\ldots,d\big\},\ \ \ k\in \N_0^d,
\end{equation*}
with $b_{k}=(b_{k_1},\ldots,b_{k_d})\in \R$, $b_{k_i}>0,$ $\ i=1,\ldots,d$. Let further $a>\frac{1}{ \min(p,q)}$. Then there exists a positive constant $C$ independent of $\{b_k\}_{k\in \N_0^d}$ such that
\begin{equation*}
\Bigg\|\Bigg\{\sum_{k\in \N_0^d} (P_{b_{k},a}f_{k})^q\Bigg\}^{1/q} \Bigg|L_p(\R)\Bigg\|\leq C\Bigg\|\Bigg\{\sum_{k\in \N_0^d} |f_{k}|^q\Bigg\}^{1/q} \Bigg|L_p(\R)\Bigg\|
\end{equation*}
holds for all systems $\{f_{k}\}_{k\in \N_0^d} \in L_p(\ell_q)$ with $\supp(\gf f_k)\subset \Omega_k$.
\end{proposition}

A very useful  relation between Peetre maximal function and differences is given by the following 
lemma which is a counterpart of \cite[Lemma 3.3.1]{Ull06}, see also  \cite[Lemma 2.3.3]{ST87B}.

 \begin{lemma}\label{onedim}\rm Let $a,b>0$, $m\in \N$, $h,k\in\re$ and $f\in \cs'(\re)$ with
	$\supp (\gf f) \subset [-b,b]$. Then there exists a constant $C>0$ independently of $f$, $b$ and $h$
	such that
	$$
	|\bigtriangleup^m_hf(x+kh)| \leq C\max(1,|bh|^a)\min(1,|bh|^m)P_{b,a}f(x)
	$$
	holds for all $x \in \re$.
\end{lemma}
\begin{proof} From the
mean value theorem we have for $\xi$ with $|\xi-x|\leq (|k|+1)h$  that
$$
|\bigtriangleup^1_hf(x+kh) |= |f'(\xi) h|\leq \sup_{|y|\leq (|k|+1)|h|}|f(x-y)| \cdot |h|.
$$ 
	By iteration we get
	\begin{equation*} 
\begin{aligned}
	|\bigtriangleup^m_h f(x+kh)|&=  |\bigtriangleup^1_h(\bigtriangleup^{m-1}_h f)(x+kh)|
	\\
	& \leq   \sup\limits_{|y|\leq (1+|k|)|h|} |\bigtriangleup^{m-1}_h f'(x-y)|\cdot |h|
	\\
	&\leq  |h|^m\sup\limits_{|y|\leq (m+|k|)|h|} |f^{(m)}(x-y)| . 
\end{aligned}
	\end{equation*}
This leads to
	\begin{equation*}
		\begin{split}
			|\bigtriangleup^m_h f(x+kh)| &\leq |h|^m\sup\limits_{y\in \re}\frac{|f^{(m)}(x-y)|}{1+|by|^a}\cdot
			\sup\limits_{|y|\leq (m+|k|)|h|}(1+|by|^a)\\
			&= |h|^m \sup\limits_{y\in \re}\frac{|f^{(m)}(x-y)|}{1+|by|^a}\cdot
			\big(1+(m+|k|)^a|bh|^a\big)\\
			&\leq 2(m+|k|)^a
			|h|^m \max(1,|bh|^a)\sup\limits_{y\in \re} \frac{|f^{(m)}(x-y)|}{1+|by|^a}.
		\end{split}
		\label{f2}
	\end{equation*}
Using \eqref{eq-Peetre}  we arrive at
\begin{equation*}
	\begin{aligned}
|\bigtriangleup^m_h f(x+kh)| & \leq 2(m+|k|)^a
|bh|^m \max(1,|bh|^a)\sup\limits_{y\in \re} \frac{|f(x-y)|}{1+|by|^a}
\\
& =2(m+|k|)^a
|bh|^m \max(1,|bh|^a) P_{b,a}f(x).
	\end{aligned}
\end{equation*}
On the other hand we have
	\begin{equation*}
		\begin{split}
			\big|\bigtriangleup^m_hf(x+kh)\big| &= \left|\sum\limits_{j=0}^m (-1)^j \binom{m}{j}f(x+(m-j+k)h)\right| \\
			&\leq C \sup\limits_{|y|\leq (m+|k|)|h|}\frac{|f(x-y)|}{1+|by|^a}(1+|by|^a)\\
			&\leq 2C\sup\limits_{y \in \re} \frac{|f(x-y)|}{1+|by|^a}\cdot\max\big(1,|(m+|k|)hb|^a\big)\\
			&\leq 2C (m+|k|)^a\max(1,|hb|^a)P_{b,a}f(x).\label{est2}
		\end{split}
	\end{equation*}
This completes the proof. 
\end{proof}
 
 Applying the above result iteratively  with respect to components in $e\subset [d]$ we get the following modified 
 version in the multivariate situation.
 
 \begin{lemma}\label{ddim-1}
 Let  $a>0$, $e\subset [d]$, $m \in \N^d$  and $h ,k\in \R$. Let further $f\in
 \mathcal{S}'(\R)$ with $\supp(\mathcal{F} f) \subset Q_{b}$, where
 $$
   Q_{b}:=[-b_1,b_1]\times \ldots \times [-b_d,b_d]\,,\ \ b_i>0,\ \ i=1, \ldots ,d.
 $$
 Then there exists a constant $C>0$ (independent of $f$, $b$, $x$
 and $h$) such that
 \begin{equation*}
 |\Delta^{{m},e}_h  f(x+k\diamond h)|
 \leq C\, \Bigg(\prod\limits_{i\in e}\, \max\big\{1,|b_ih_i|^a\big\}\, \min\big\{1,|b_ih_i|^{m_i}\big\}
 \Bigg)\,  P_{b,a} f(x)
 \end{equation*}
 holds for all $x\in \R$. 
 \end{lemma}


\section{Results and proofs}\label{proofs}
For a quasi-Banach space $X$ of functions  we shall call a function $g$ a pointwise multiplier
if $g \, \cdot \, f \in X$ for all $f \in X$.
By $M(X)$ we denote the set of all pointwise multipliers for $X$, i.e.,
\[
M(X):= \big\{g:~ g  \cdot   f \in X \quad \forall f\in X\big\}
\]
We call $X$ an algebra with respect to pointwise multiplication 
(for short a multiplication algebra) if $f  \cdot  g \in X$ for all $f,g\in X$ and there exist a constant $C>0$ such that 
\begin{equation*}
\| f\cdot g\,|\, X\| \leq C\|\,f\,|\,X\| \cdot \|\, g\,|\,X\|
\end{equation*}
holds for all $f,g\in X$. It is obvious that if $X$ is a multiplication algebra we have, $X\subset M(X)$.

Let us recall the algebra property of the Sobolev spaces with mixed smoothness. 
\begin{theorem}\label{need}
	Let $m\in \N$ and $1\leq p<\infty$. Then the space $S^m_pW(\R)$ is a multiplication algebra.
\end{theorem}
The proof of the above theorem can be found in \cite{NS1701,Ng25}. Our    result  with respect to  Triebel--Lizorkin spaces of dominating mixed smoothness reads as follows.

\begin{theorem}\label{main-be}
Let $0<p<\infty$, $0<p\leq q\leq \infty$ and $r >\frac{1}{p}$. Then $S^r_{p,q}F(\R)$ is a multiplication algebra.
\end{theorem}

Proof of the above theorem needs the following lemmas which can be found in \cite[Theorem 2.4.2]{ST87B} and \cite{Vy06,HV09}. 
By $C(\R)$ we denote the collection of all uniformly continuous and 
bounded functions $f: ~ \R \to \com$, equipped with the sup-norm.




\begin{lemma}\label{emb1}
	Let $0<p<\infty$, $0< q\leq \infty$ and $r \in \re$. Then $S^r_{p,q}F(\R)$ is continuously embedded into $C(\R)$ if and only if 
	\begin{itemize}
		\item either  $r>1/p$
		\item or $0< p \leq 1$, $r=1/p$.
	\end{itemize}  
\end{lemma}

\begin{lemma}
	\label{trace} Let $0< p< \infty$, $0<q\leq \infty$ and $r>1/p$. 
	Let further $L\in \N$ and $ L\leq d$. 
	If    $f \in S^r_{p,q}F(\R)$ then the function 
	\[
	g(x_1,\ldots,x_L):=f(x_1,\ldots,x_L,x_{L+1},\ldots,x_d)
	\] of the $L$ variables  $x_1,\ldots, x_L$ ($x_{L+1},\ldots,x_d$ are considered as fixed) belongs to the space $ S^r_{p,q}F(\re^L)$.
\end{lemma}

Now we are in the position to prove Theorem \ref{main-be}.
\begin{proof}Let $m\in \N$ such that $r <m\leq r+1$. We denote
$$
\overline{m}=(m,\ldots,m) \in \N^d.
$$
Since the norm in Proposition \ref{diff} 
does not depend on the order of difference whose components are larger than $r$, in the sense of equivalent norms, we shall prove that  
\begin{equation*}
     \|\, f g\,|\,S^r_{p,q} F(\R)\|_{(2\overline{m})} \leq C \,  \| \, f\, |S^r_{p,q}F(\R) \| \cdot  \|\,  g \, |S^r_{p,q}F(\R) \|
\end{equation*} 
holds for all $f,g\in S^r_{p,q}F(\R)$ with the positive constant $C$ independent of $f$ and $g$. 

In this proof, when we write $A\lesssim B$ we mean that there is a  positive constant $C$ which may depend on parameters $d$, $r$, $m$, $p$ and $q$ but is independent of $f$ and $g$ such that $A\le C\cdot B$. 

\noindent
{\it Step 1.} Taking Lemma \ref{emb1} into account  we obtain that
\begin{equation*}
\| f g\,|L_p(\R)\| \, \leq\, \| f|L_p(\R)\|\cdot \| g|C(\R)\|\, \lesssim    \| f|S^r_{p,q}F(\R)\|\cdot \|g|S^r_{p,q}F(\R)\|.
\end{equation*}  
This inequality should  be interpreted as the estimate needed for the term with $e= \emptyset$.
Next we need some identities for differences.
Note that if $\psi,\, \phi:\ \re \to \com$ and $m \in \N$ we have
\be\label{formular}
\Delta_h^{m}(\psi \phi)(t)=\sum_{j=0}^{m} \binom{m}{j} \, \Delta_h^{m-j}\psi(t+j h)\, \Delta_h^{j}\phi(t),\qquad t,h\in \re \, ,
\ee 
which can be proved by induction on $m$.
Let $e\subset [d]$, $e\not=\emptyset$ and recall the notation 
\[
x \diamond y = (x_1y_1, \ldots , x_d y_d)\in \R\, 
\]
and
\[
\N_0^d(e) = \big\{k\in \N_0^d: \ k_i=0\ \text{if}\ i\not \in e\big\}.
\]
Then we derive from \eqref{formular} that
\begin{equation*}\label{mot}
\Delta_{h}^{2{\overline{m}},e}(f  g)(x)=\sum_{ u \in \N_0^d(e)\atop |u|_\infty\leq 2m} \binom{2\overline{m}}{u}\, 
\Delta_{h}^{2 \overline{m} - u ,e}f(x+ u  \diamond h)\, \Delta_{h}^{ u ,e}g(x)\,,\quad\, x, h\in \R\, , 
\end{equation*}
holds. Here $2\overline{m}-u :=(2m-u_1, \ldots ,2m-u_d)$ and 
\[
\binom{2\overline{m}}{u} = \prod_{i \in e} \binom{2m}{u_i} \, .
\]
Hence
$$
\|fg | S^r_{p,q}F(\R)\|_{(2\overline{m})} \lesssim \Bigg(
\| f|S^r_{p,q}F(\R)\|\cdot \|g|S^r_{p,q}F(\R)\|+\sum_{e\subset [d], e\not=\emptyset} \sum_{ u \in \N_0^d(e)\atop |u|_\infty\leq 2m}S_{e,u}\Bigg),
$$
where
\begin{equation*}\label{ws-117}
S_{e,u} :=\Bigg\|\Bigg\{\sum\limits_{k \in \N_0^d(e)} 2^{r|k|_1 q}\Bigg(\sup_{|h_i| < 2^{-k_i}\atop i \in e} 
\big| \Delta_{h}^{2 \overline{m} - u ,e}f(\cdot+ u  \diamond h)\Delta_{h}^{ u ,e}g(\cdot)\big|\Bigg)^q\Bigg\}^{1/q}\Bigg|L_p(\R) \Bigg\|.
\end{equation*}
The main step of the proof will consist  in the  estimates of the terms $S_{e,u}$. 
Therefore we have to  consider different cases. Without loss of generality we may assume that $$e=\{1, \ldots ,N\}$$ for some natural number $N$, $N \le d$. 
\\
{\it Step 2.} The case $u_i\geq m$ for all $ i\in e$. 
We have
\begin{equation*}
	\begin{aligned}
 \big| \Delta_{h}^{2\overline{m} -u ,e}f(x+u  \diamond h)\Delta_{h}^{u ,e}g(x)\big|
&\le     \sup_{x\in \R}\big | \Delta_{h}^{2\overline{m} -u ,e}f(x+u  \diamond h)\big| \cdot   |\Delta_{h}^{u ,e}g(x)|\\
& \lesssim  \| f |C(\R)\|\cdot  \big | \Delta_{h}^{u ,e}g(x) \big |\,. 		
	\end{aligned}
\end{equation*}
The embedding $ S^r_{p,q}F(\R) \hookrightarrow C(\R)$, see Lemma \ref{emb1},    implies 
\begin{equation*}
	\begin{aligned}
 \big| \Delta_{h}^{2\overline{m} -u ,e}f(x+u  \diamond h)\Delta_{h}^{u ,e}g(x)\big| 
 \lesssim \big\|f |S^r_{p,q}F(\R)\big\|\cdot   \big| \Delta_{h}^{u ,e}g(x) \big|.
	\end{aligned}
\end{equation*}
Consequently we have 
\begin{equation*}
\begin{aligned}
	 S_{e,u}
	 & \lesssim \big\|f \big|S^r_{p,q}F(\R)\big\|  \cdot  \Bigg\|\Bigg\{\sum\limits_{k  \in \N_0^d(e)} 2^{r|k |_1 q}\, \sup_{|h_i| < 2^{-k_i}\atop i \in e}  \big|\Delta_{h}^{u ,e}g(\cdot) \big|^{q} \Bigg\}^{1/q}\bigg|L_p(\R)  \Bigg\| 
	 \\
	 	 & \lesssim \big\|f \big|S^r_{p,q}
	 F(\R)\big\|\cdot  \big\|g \big|S^r_{p,q}F(\R)\big\|_{(u)}
	 	 \\
	 & \lesssim \big\|f \big|S^r_{p,q}
	 F(\R)\big\|\cdot  \big\|g \big|S^r_{p,q}F(\R)\big\|\, .
\end{aligned}
\end{equation*}
{\it Step 3.} The case $u_i \le m$ for all $i\in e$. Obviously we have  $2m-u_i \ge m $ for all $i \in e$. In this case we have
\begin{equation*}
	\begin{aligned}
		\big| \Delta_{h}^{2\overline{m} -u ,e}f(x+u  \diamond h)\Delta_{h}^{u ,e}g(x)\big|
		&\le     \big | \Delta_{h}^{2\overline{m} -u ,e}f(x+u  \diamond h)\big| \,   \sup_{x\in \R}|\Delta_{h}^{u ,e}g(x)|\\
		& \lesssim \| g |C(\R)\|\cdot  \big | \Delta_{h}^{2\overline{m} -u ,e}f(x+u  \diamond h)\big|
		\\
		& \lesssim \| g |S^r_{p,q}F(\R)\|\cdot  \big | \Delta_{h}^{2\overline{m} -u ,e}f(x+u  \diamond h)\big| 
	\end{aligned}
\end{equation*}
which implies
\begin{equation*}
	\begin{aligned}
		S_{e,u}
		& \lesssim \big\|g \big|S^r_{p,q}F(\R)\big\|  \cdot \Bigg\|\Bigg\{\sum\limits_{k  \in \N_0^d(e)} 2^{r|k |_1 q}  \sup_{|h_i|<2^{-k_i}\atop i\in e} \big | \Delta_{h}^{2\overline{m} -u ,e}f(\cdot+u  \diamond h)\big|^q \Bigg\}^{1/q} \Bigg| L_p(\R) \Bigg\|.
	\end{aligned}
\end{equation*}
Denote  the second factor on the right-hand side by $F_{e,u}$. 
Next we apply  the tensor product system $(\varphi_{j})_{j\in \N_0^d}$, defined in \eqref{ws-101}.  
We shall use the convention that in the univariate case $\varphi_n \equiv 0$ if $n< 0$, which implies that 
$\varphi_{(j_1, \ldots  , j_d)} \equiv 0$ if $\min_i j_i < 0$. For any $k \in \N_0^d(e)$
this yields 
\[
f(x) = \sum_{ \ell \in \zz^{d}}  \gf^{-1} [ \varphi_{k+ \ell } \gf f](x) =\sum_{\ell\in \Z}f_{k+\ell}(x)
\]
with convergence in $ S^r_{p,q}F(\re^d)$ and therefore in $C(\re^d)$, see Lemma \ref{emb1}. Here we denote
$$
f_{k+\ell}(x)=\gf^{-1} [ \varphi_{k+ \ell } \gf f](x)\,.
$$
Then we get
\[
\Delta_{h}^{2\overline{m} -u ,e}f(x+u\diamond h) =  \sum_{\ell\in \Z}\Delta_{h}^{2\overline{m} -u ,e}f_{k+\ell}(x+u\diamond h).
\]
It is easily seen that
\begin{equation}\label{eq-sum-01}
\Delta_{h}^{2\overline{m} -u ,e}f_{k+\ell}(x+u  \diamond h) = \sum_{\eta\in U}C_\eta\Delta_{h}^{\overline{m} ,e}f_{k+\ell}(x_1+\eta_1h_1,\ldots,x_N+\eta_N h_N,x_{N+1},\ldots,x_d)	 
\end{equation}
for some finite set $U\subset \N_0^N$ and $C_\eta\in \re$. 
We need one more notation. For $\ell \in \Z$ we put
\[
\omega (\ell):= \big\{i \in \{1, \ldots , d\}:~ \ell_i <0 \big\} \qquad \mbox{and}\qquad 
\overline{\omega} (\ell):= \big\{i \in \{1, \ldots , d\}:~ \ell_i \ge 0 \big\}\, .
\]
It is easily seen that
$$
\omega(\ell)\subset e\ \ \ \text{and}\ \ \ e_0=[d]\backslash e\subset \overline{\omega}(\ell).
$$
By writing 
\[
\Delta_{h}^{  \overline{m}  ,e} = \Bigg(\prod_{i \in \overline{\omega} (\ell) \cap e} 
\Delta_{h_i}^{  m }\Bigg) \Bigg(\prod_{i \in \omega (\ell)   } \Delta_{h_i}^{  m  }\Bigg)\, ,
\] with $\frac{1}{p}<a<r$ we get from Lemma \ref{ddim-1}
$$
\big|\Delta_{h}^{\overline{m} ,e}f_{k+\ell}(x_1+\eta_1h_1,\ldots,x_N+\eta_N h_N,x_{N+1},\ldots,x_d) \big|\lesssim \prod_{i\in \omega(\ell)} 2^{\ell_im} \prod_{i\in \overline{\omega}(\ell)\cap e} 2^{\ell_i a} P_{2^{k+\ell},a}f_{k+\ell}(x)
$$
and therefore
$$
\big|\Delta_{h}^{2\overline{m} -u ,e}f_{k+\ell}(x+u  \diamond h) \big|\lesssim \prod_{i\in \omega(\ell)} 2^{\ell_im} \prod_{i\in \overline{\omega}(\ell)\cap e} 2^{\ell_i a} P_{2^{k+\ell},a}f_{k+\ell}(x).
$$
Inserting this into \eqref{eq-sum-01} we can estimate $F_{e,u}$ as follows
\begin{equation*}
	\begin{aligned}
 F_{e,u}  & \lesssim \Bigg\|\Bigg\{\sum\limits_{k  \in \N_0^d(e)} \Bigg(\sum_{\ell\in \Z}2^{r|k |_1  } \prod_{i\in \omega(\ell)} 2^{\ell_im } \prod_{i\in \overline{\omega}(\ell)\cap e} 2^{\ell_i a } P_{2^{k+\ell},a}f_{k+\ell}(\cdot)\Bigg)^q \Bigg\}^{1/q} \Bigg| L_p(\R) \Bigg\|
 \\
&\lesssim \Bigg\|\Bigg\{\sum\limits_{k  \in \N_0^d(e)} 2^{r|k+\ell |_1 q }\Bigg(\sum_{\ell\in \Z} \prod_{i\in \omega(\ell)} 2^{\ell_im } \prod_{i\in \overline{\omega}(\ell)\cap e} 2^{\ell_i a } \prod_{i=1}^d 2^{-\ell_ir} P_{2^{k+\ell},a}f_{k+\ell}(\cdot)\Bigg)^q \Bigg\}^{1/q} \Bigg| L_p(\R) \Bigg\|.
	\end{aligned}
\end{equation*}
If $1\leq q\leq \infty$ we have
\begin{equation*}
	\begin{aligned}
		F_{e,u}  
		&\lesssim \Bigg\|\sum_{\ell\in \Z} \prod_{i\in \omega(\ell)} 2^{\ell_im } \prod_{i\in \overline{\omega}(\ell)\cap e} 2^{\ell_i a } \prod_{i=1}^d 2^{-\ell_ir}\Bigg\{\sum\limits_{k  \in \N_0^d(e)} 2^{r|k+\ell |_1 q }\Big( P_{2^{k+\ell},a}f_{k+\ell}(\cdot)\Big)^q \Bigg\}^{1/q} \Bigg| L_p(\R) \Bigg\|
\\
		&\lesssim \Bigg\|\sum_{\ell\in \Z}\prod_{i\in \omega(\ell)} 2^{\ell_i(m-r) } \prod_{i\in \overline{\omega}(\ell)\cap e} 2^{-\ell_i (r-a) } \prod_{i\in e_0} 2^{-\ell_ir}\Bigg\{\sum\limits_{k  \in \N_0^d(e)} \Big(2^{r|k+\ell |_1  }  P_{2^{k+\ell},a}f_{k+\ell}(\cdot)\Big)^q \Bigg\}^{1/q} \Bigg| L_p(\R) \Bigg\|.
	\end{aligned}
\end{equation*}
Denoting
$$
\delta=\min\{m-r, r-a,r \}>0
$$
we get
\begin{equation*}
	\begin{aligned}
		F_{e,u}  
		&\lesssim\Bigg\|\sum_{\ell\in \Z}\prod_{i=1}^d  2^{-|\ell|_1\delta}\Bigg\{\sum\limits_{k  \in \N_0^d} \Big(2^{r|k|_1  }  P_{2^k,a}f_{k}(\cdot)\Big)^q \Bigg\}^{1/q} \Bigg| L_p(\R) \Bigg\|
\\
&\lesssim \Bigg\| \Bigg\{\sum\limits_{k  \in \N_0^d} \Big(2^{r|k|_1  }  P_{2^k,a}f_{k}(\cdot)\Big)^q \Bigg\}^{1/q} \Bigg| L_p(\R) \Bigg\|
\\
&	\lesssim 	 \big\| f\big| S^r_{p,q}F(\R) \big\|.
	\end{aligned}
\end{equation*}
Here in the last estimate we used Proposition \ref{peetremax}. Consequently
\begin{equation}\label{eq-want}
	\begin{aligned}
	S_{e,u}
& \lesssim \big\|g \big|{S^r_{p,q}}
F(\R)\big\|\cdot \big\|f \big|{S^r_{p,q}}
F(\R)\big\|.
	\end{aligned}
	\end{equation}
If $0<q<1$ we have
\begin{equation*}
	\begin{aligned}
		F_{e,u}  
		&\lesssim \Bigg\|\Bigg\{\sum\limits_{k  \in \N_0^d(e)} 2^{r|k+\ell |_1 q }\sum_{\ell\in \Z}\Bigg( \prod_{i\in \omega(\ell)} 2^{\ell_im } \prod_{i\in \overline{\omega}(\ell)\cap e} 2^{\ell_i a } \prod_{i=1}^d 2^{-\ell_ir} P_{2^{k+\ell},a}f_{k+\ell}(\cdot)\Bigg)^q \Bigg\}^{1/q} \Bigg| L_p(\R) \Bigg\|
		\\
		&\lesssim \Bigg\|\Bigg\{\sum_{\ell\in \Z} \bigg(\prod_{i\in \omega(\ell)} 2^{\ell_im } \prod_{i\in \overline{\omega}(\ell)\cap e} 2^{\ell_i a } \prod_{i=1}^d 2^{-\ell_ir}\bigg)^q\sum\limits_{k  \in \N_0^d(e)}\Big( 2^{r|k+\ell |_1  } P_{2^{k+\ell},a}f_{k+\ell}(\cdot)\Big)^q \Bigg\}^{1/q} \Bigg| L_p(\R) \Bigg\|
		\\
			&\lesssim \Bigg\|\Bigg\{\sum_{\ell\in \Z}  2^{-|\ell|_1\delta q}\sum\limits_{k  \in \N_0^d}\Big( 2^{r|k|_1  } P_{2^k,a}f_{k}(\cdot)\Big)^q \Bigg\}^{1/q} \Bigg| L_p(\R) \Bigg\|
			\\
			&\lesssim \|f|S^r_{p,q}F(\R)\|
	\end{aligned}
\end{equation*}
which implies \eqref{eq-want}.

\noindent
{\it Step 4.} The remaining cases.  Without loss of generality  
  we assume in addition that
\[
u=(u_{1}, \ldots ,u_L,u_{L+1}, \ldots ,u_N,0,\ldots,0)
\] 
with
$$ m\leq u_i\leq 2m,\quad i=1, \ldots ,L,\qquad 0\leq u_i<m, \quad i= L+1, \ldots ,N$$
and $1\leq L\le N$ and $L< d$. 
For brevity we put
\[
e_1:= \{L+1,\ldots,N\} \qquad \mbox{ and } \qquad  e_2:=\{1,\ldots,L \}\, . 
\]
By assumption both sets are nontrivial.
This covers all remaining cases up to an enumeration.
We write as
\begin{equation*}
\begin{aligned}
\Delta_{h}^{2\overline{m} -u ,e}f(x+u  \diamond h) 
&=\sum_{\eta \in U} C_\eta\Delta_{h}^{\overline{m} ,e_1}f(x_1+\eta_{1}h_{1},\ldots, x_N+\eta_Nh_N,x_{N+1},\ldots,x_d ),
\end{aligned}
\end{equation*}
where $U\subset \N_0^N$ is a finite set and $C_\eta\in \re$.
We have by Proposition \ref{trace}
\begin{equation*}
	\begin{aligned}
&\big|\Delta_{h}^{2\overline{m} -u ,e}f(x+u  \diamond h)\big|
\\
& \leq \sum_{\eta \in U} |C_\eta| \sup_{x_1,\ldots,x_L} |\Delta_{h}^{\overline{m} ,e_1}f(x_1+\eta_{1}h_{1},\ldots, x_N+\eta_Nh_N,x_{N+1}\ldots,x_d ) |
\\
&\lesssim \sum_{\eta \in U}  \big\|\Delta_{h}^{\overline{m} ,e_1}f(\,\cdot\,,x_{L+1}+\eta_{L+1}h_{L+1},\ldots, x_N+\eta_Nh_N,x_{N+1},\ldots,x_d ) \big|{S^r_{p,q}F(\re^L)}\big\|.
\end{aligned}
\end{equation*}
Denoting
$$
z_f=(y_1,\ldots,y_L,x_{L+1}+\eta_{L+1}h_{L+1},\ldots, x_N+\eta_Nh_N,x_{N+1},\ldots,x_d )
$$
we get
\begin{equation*}
	\begin{aligned}
		&\big|\Delta_{h}^{2\overline{m} -u ,e}f(x+u  \diamond h) \big|
		\\
		&\lesssim \sum_{\eta \in U} \sum_{e_2'\subset e_2}\Bigg[ \int\limits_{\re^L}\Bigg(\sum_{\mu\in \N_0^L(e_2')}2^{r|\mu|_1q}\sup_{|s_i|<2^{-\mu_i}\atop i\in e_2'}\big|\Delta_s^{\overline{m},e_2'}\Delta_{h}^{\overline{m} ,e_1}f(z_f)\big|^q \Bigg)^{p/q}\dx y_1\ldots \dx y_L\Bigg]^{1/p}.
	\end{aligned}
\end{equation*}
Similarly, with
$$
z_g=(x_1,\ldots,x_L,y_{L+1},\ldots,y_d),
$$
for some finite subset $V\subset \N_0^d(e_1)$ and $C_\theta\in \re$, $\theta\in V$, we have
\begin{equation*}
	\begin{aligned}
	\big|	\Delta_{h}^{u ,e}g(x)\big|
		&\leq \sum_{\theta\in V} |C_\theta| \big|\Delta_{h}^{u ,e_2}g(x_1,\ldots,x_L, x_{L+1}+\theta_{L+1}h_{L+1},\ldots,x_N+\theta_Nh_N,x_{N+1},\ldots,x_d )\big|
		\\
&\lesssim \sum_{\theta\in V} \sup_{x_{L+1},\ldots,x_d}|\Delta_{h}^{u ,e_2}g(x_1,\ldots,x_L, x_{L+1}+\theta_{L+1}h_{L+1},\ldots,x_N+\theta_Nh_N,x_{N+1},\ldots,x_d )\big|
		\\
		&\lesssim \big\|\Delta_{h}^{u ,e_2}g(x_1,\ldots,x_L,\, \cdot\,)  \big| S^r_{p,q}(\re^{d-L}) \big\|
		\\
		&\lesssim \sum_{e_1'\subset e_1\cup e_0}\Bigg[ \int\limits_{\re^{d-L}}\Bigg(\sum_{\nu\in \N_0^L(e_1')}2^{r|\nu|_1q}\sup_{|t_i|<2^{-\nu_i}\atop i\in e_1'}\big|\Delta_t^{\overline{m},e_1'}\Delta_{h}^{u ,e_2}g(z_g)\big|^q \Bigg)^{p/q}\dx y_{L+1}\ldots \dx y_d\Bigg]^{1/p}.
	\end{aligned}
\end{equation*}
We have
\begin{equation*}
	\begin{aligned}
&		S_{e,u}
 \lesssim \sum_{\eta \in U} \sum_{e_2'\subset e_2}\sum_{e_1'\subset e_1\cup e_0}     \Bigg\|\Bigg\{\sum\limits_{k  \in \N_0^d(e)} 2^{r|k |_1 q}  
 \\
 &\times \sup_{|h_i|<2^{-k_i}\atop i\in e}\Bigg[ \int\limits_{\re^L}\Bigg(\sum_{\mu\in \N_0^L(e_1')}2^{r|\mu|_1q}\sup_{|s_i|<2^{-\mu_i}\atop i\in e_2'}\big|\Delta_s^{\overline{m},e_2'}\Delta_{h}^{\overline{m} ,e_1}f(z_f)\big|^q \Bigg)^{p/q}\dx y_1\ldots \dx y_L\Bigg]^{q/p}
\\
&\times \Bigg[ \int\limits_{\re^{d-L}}\Bigg(\sum_{\nu\in \N_0^L(e_1')}2^{r|\nu|_1q}\sup_{|t_i|<2^{-\nu_i}\atop i\in e_1'}\big|\Delta_t^{\overline{m},e_1'}\Delta_{h}^{u ,e_2}g(z_g)\big|^q \Bigg)^{p/q}\dx y_{L+1}\ldots \dx y_d\Bigg]^{q/p}\Bigg\}^{1/q}\Bigg|L_p(\R)   \Bigg\|
\\
&
\lesssim \sum_{\eta \in U} \sum_{e_2'\subset e_2}\sum_{e_1'\subset e_1\cup e_0}     \Bigg\|\Bigg\{\sum\limits_{k  \in \N_0^d(e)} 2^{r|k |_1 q}  
 \sup_{|h_i|<2^{-k_i}\atop i\in e}\Bigg[ \int\limits_{\re^d}\Bigg(\sum_{\mu\in \N_0^L(e_1')}2^{r|\mu|_1q}\sup_{|s_i|<2^{-\mu_i}\atop i\in e_2'}\big|\Delta_s^{\overline{m},e_2'}\Delta_{h}^{\overline{m} ,e_1}f(z_f)\big|^q \Bigg)^{p/q}
\\
&\times \Bigg(\sum_{\nu\in \N_0^L(e_1')}2^{r|\nu|_1q}\sup_{|t_i|<2^{-\nu_i},i\in e_1'}\big|\Delta_t^{\overline{m},e_1'}\Delta_{h}^{u ,e_2}g(z_g)\big|^q \Bigg)^{p/q}\dx y\Bigg]^{q/p}\Bigg\}^{1/q}\Bigg|L_p(\R)   \Bigg\|\,.
	\end{aligned}
\end{equation*}
Applying Generalized inequality with $q\ge p$ we get
\begin{equation*}
	\begin{aligned}
	 S_{e,u}
		  &\lesssim \sum_{e_2'\subset e_2}\sum_{e_1'\subset e_1\cup e_0} \sum_{\eta\in U}    \Bigg\|\Bigg\{\int\limits_{\re^d}\Bigg(\sum\limits_{k  \in \N_0^d(e)} 2^{r|k |_1 q}  
 \sum_{\mu\in \N_0^L(e_1')}2^{r|\mu|_1q}	\sup_{|h_i|<2^{-k_i}\atop i\in e_1}\sup_{|s_i|<2^{-\mu_i}\atop i\in e_2'}\big|\Delta_s^{\overline{m},e_2'}\Delta_{h}^{\overline{m} ,e_1}f(z_f)\big|^q \Bigg)^p
	 \\
	 &
	\times	\Bigg(\sum_{\nu\in \N_0^L(e_1')}2^{r|\nu|_1q}	\sup_{|h_i|<2^{-k_i}\atop i\in e_2}\sup_{|t_i|<2^{-\nu_i}\atop i\in e_1'}\big|\Delta_t^{\overline{m},e_1'}\Delta_{h}^{u ,e_2}g(z_g)\big|^q \Bigg)^p\dx y\Bigg\}^{1/p}\Bigg| L_p(\R)   \Bigg\|
	\\
		&
\lesssim \sum_{e_2'\subset e_2}\sum_{e_1'\subset e_1\cup e_0}\sum_{\eta\in U}      \Bigg\|\Bigg\{\int\limits_{\re^d}\Bigg(\sum\limits_{k^1  \in \N_0^d(e_1)}  
\sum_{\mu\in \N_0^L(e_1')}2^{r(|k^1|_1+|\mu|_1)q}	\sup_{|h_i|<2^{-k_i}\atop i\in e_1}\sup_{|s_i|<2^{-\mu_i}\atop i\in e_2'}\big|\Delta_s^{\overline{m},e_2'}\Delta_{h}^{\overline{m} ,e_1}f(z_f)\big|^q \Bigg)^p
\\
&
\times	\Bigg(\sum\limits_{k^2  \in \N_0^d(e_2)}\sum_{\nu\in \N_0^L(e_1')}2^{r(|k^2|_1+|\nu|_1)q}	\sup_{|h_i|<2^{-k_i}\atop i\in e_2}\sup_{|t_i|<2^{-\nu_i}\atop i\in e_1'}\big|\Delta_t^{\overline{m},e_1'}\Delta_{h}^{u ,e_2}g(z_g)\big|^q \Bigg)^p\dx y\Bigg\}^{1/p}\Bigg| L_p(\R)   \Bigg\|.
	\end{aligned}
\end{equation*}
By changing the notation we find that
\begin{equation*}
	\begin{aligned}
&	S_{e,u}
	  \lesssim \sum_{e_2'\subset e_2}\sum_{e_1'\subset e_1\cup e_0}\sum_{\eta \in U}    \|g|S^r_{p,q}F(\R)\|
	\cdot  \Bigg\|\Bigg\{\sum\limits_{k  \in \N_0^d(e_1\cup e_1')}  
	 2^{r|k|_1q}	
	 \\
	 & \sup_{|h_i|<2^{-k_i}\atop i\in e_1\cup e_2'}\big| \Delta_{h}^{\overline{m} ,e_1\cup e_2'}f(x_1,\ldots,x_L,x_{L+1}+\eta_{L+1}h_{L+1},\ldots, x_N+\eta_Nh_N,x_{N+1},\ldots,x_d )\big|^q\Bigg\}^{1/q}\Bigg| L_p(\R) \Bigg\|. 
	\end{aligned}
\end{equation*}
Next following the same argument as Step 3 we finally get
\begin{align*}
	S_{e,u}
&\lesssim \sum_{e_2'\subset e_2}\sum_{e_1'\subset e_1\cup e_0}\sum_{\eta \in U}    \|g|S^r_{p,q}F(\R)\| \cdot \|f|S^r_{p,q}F(\R)\|
\\
& \lesssim \|g|S^r_{p,q}F(\R)\| \cdot \|f|S^r_{p,q}F(\R)\|.	 
\end{align*}
The proof is complete.
\end{proof}

By multi-linear interpolation we can extend the   results in Theorem \ref{need} to  the following. 
\begin{theorem}\label{algebra}
	Let $1< p <\infty$ and  $r>\max(1/p,1-1/p)$. Then the space $S^r_p H(\R)$ is a multiplication algebra.
\end{theorem}
To prove Theorem \ref{algebra} we need multi-linear interpolation in the sense of Calder\'on \cite{Cal64}. Let $X$ and $Y$ be Banach spaces. By $[X,Y]_{\Theta}$, $\Theta\in (0,1)$, we denote the classical complex interpolation method of Calder\'on, see \cite{BL76B,Lu09B,Tr78} for details.
\begin{proposition} \label{mul}
	Let $(A_i,B_i)$, $i=1,\ldots,n$, and $(A,B)$ be interpolation pairs. Let $L(x_1,\ldots,x_n)$, $x_i\in A_i\cap B_i$ be a multilinear operator defined in the direct sum $ \bigoplus_{i=1}^n(A_i\cap B_i)$ with the value in $A\cap B$ such that
$$
	\| L(x_1,\ldots,x_n)|A\| \leq M_A\prod_{i=1}^n \|x_i|A_i\|
$$
and
$$
	\| L(x_1,\ldots,x_n)|B\| \leq M_B\prod_{i=1}^n \|x_i|B_i\|.
$$
	Then $L$ can be extended continuously to a multilinear mapping from $\bigoplus_{i=1}^n [A_i,B_i]_{\Theta}$ into $[A,B]_{\Theta}$ with 
\begin{equation*}
	\big\| L(x_1,\ldots,x_n)\big|[A,B]_{\Theta}\big\| \leq M_A^{1-\Theta}M_B^{\Theta}\prod_{i=1}^n \big\|x_i\big|[A_i,B_i]_{\Theta}\big\|.
\end{equation*}
\end{proposition}
\begin{proof}[Proof of Theorem \ref{algebra}]  If $r\in \N$, the assertion was already proved,  see Theorem \ref{need}. The case $1<p\leq 2$ is a consequence of Theorem \ref{main-be} and Proposition \ref{littlewood} .
	We consider the case $2< p <\infty$. Let $1-\frac{1}{p}<r<1$. There exits $\Theta\in (0,1)$, $r_1>1/2$ and $p_0\in (1,2)$ such that
	\begin{equation*}
	r= (1-\Theta)+\Theta r_1,\qquad \text{and}\qquad \frac{1}{p}=\frac{1-\Theta}{p_0}+\frac{\Theta}{2}.
	\end{equation*}
	It follows that
	\begin{equation*}
	S^r_pH(\R) = \big[S^1_{p_0}H(\R),S^{r_1}_pH(\R)\big]_{\Theta},
	\end{equation*}
	see \cite[Theorem 1.7]{Ull06} or Vybiral \cite[Theorem 4.6]{Vy06}. Now Proposition  \ref{mul} and Theorem \ref{need}   imply the result in this case. Similar argument holds for $r\in (m,m+1)$, $m\in \N$. The proof is completed.
\end{proof}


\noindent
{\bf Acknowledgments.} This research is funded by the Vietnam National Foundation for Science and Technology Development (NAFOSTED) in the frame of the NAFOSTED-SNSF Joint Research Project under Grant IZVSZ2\_229568.


\bibliographystyle{abbrv}

\bibliography{AllBib}

\begin{thebibliography}{10}

\bibitem{BH21}
A.~Behzadan and M.~Holst.
\newblock {Multiplication in Sobolev Spaces, Revisited }.
\newblock {\em Ark. Mat.}, 59:275--306, 2021.

\bibitem{BL76B}
J.~Bergh and J.~L\"ofstr\"om.
\newblock {\em {Interpolation Spaces. An Introduction}}.
\newblock Lect. Notes, Scuola Normale Superiore PisaSpringer, New York, 1976.

\bibitem{BuGr04}
H.-J. Bungartz and M.~Griebel.
\newblock Sparse grids.
\newblock {\em Acta Numer.}, 13:147--269, 2004.

\bibitem{Cal64}
A.~P. Calder\'on.
\newblock {Intermediate spaces and interpolation}.
\newblock {\em Studia Math.}, 24:113--190, 1964.

\bibitem{Han10}
M.~Hansen.
\newblock {\em Nonlinear Approximation and Function Spaces of Dominating Mixed
  Smoothness}.
\newblock PhD thesis, University of Jena, 2010.

\bibitem{HV09}
M.~Hansen and J.~Vybiral.
\newblock {The Franke-Jawerth embedding of spaces with dominating mixed
  smoothness}.
\newblock {\em Georgian Math. J.}, 19:667--682, 2009.

\bibitem{LSYY24}
Z.~Li, W.~Sickel, D.~Yang, and W.~Yuan.
\newblock {Pointwise multipliers for Besov spaces
  $B^{0,b}_{p,\infty}(\mathbb{R}$ with only logarithmic smoothness}.
\newblock {\em Ann. Mat. Pura Appl.}, 203:703–763, 2024.

\bibitem{Lu09B}
A.~Lunardi.
\newblock {\em {Interpolation theory}}.
\newblock Lect. Notes, Scuola Normale Superiore Pisa, 2009.

\bibitem{MS85B}
V.~Maz’ya and T.~Shaposhnikova.
\newblock {\em {Theory of Multipliers in Spaces of Differentiable Functions}}.
\newblock Pitman,Boston, MA, 1985.

\bibitem{MS09B}
V.~Maz’ya and T.~Shaposhnikova.
\newblock {\em {Theory of Sobolev Multipliers with Applications to Differential
  andIntegral Operators}}.
\newblock Springer, Berlin, 2009.

\bibitem{Netrusov95}
Y.~V. Netrusov.
\newblock {Theorems on traces and multipliers for functions in
  Lizorkin–Triebel spaces}.
\newblock {\em J. Math. Sci.}, 77:3221--3224, 1995.

\bibitem{Ng25}
V.~K. Nguyen.
\newblock {Optimality conditions for change of variable operators in Sobolev
  spaces with mixed derivatives}.
\newblock {\em Anal. Math.}, to appear.

\bibitem{NS1702}
V.~K. Nguyen and W.~Sickel.
\newblock {Pointwise multipliers for Besov spaces of dominating mixed
  smoothness -- II}.
\newblock {\em Sci. China Math.}, 60:2241--2262, 2017.

\bibitem{NS1701}
V.~K. Nguyen and W.~Sickel.
\newblock {Pointwise multipliers for Sobolev and Besov spaces of dominating
  mixed smoothness}.
\newblock {\em J. Math. Anal. Appl.}, 452:62--90, 2017.

\bibitem{NoWo08}
E.~Novak and H.~Wo{\'z}niakowski.
\newblock {\em {Tractability of Multivariate Problems, Volume I: Linear
  Information}}.
\newblock EMS Tracts in Mathematics, Vol. 6, Eur. Math. Soc. Publ. House,
  Z\"urich, 2008.

\bibitem{NoWo10}
E.~Novak and H.~Wo{\'z}niakowski.
\newblock {\em {Tractability of Multivariate Problems, Volume II: Standard
  Information for Functionals}}.
\newblock EMS Tracts in Mathematics, Vol. 12, Eur. Math. Soc. Publ. House,
  Z\"urich, 2010.

\bibitem{NoWo12}
E.~Novak and H.~Wo{\'z}niakowski.
\newblock {\em {Tractability of Multivariate Problems: Volume III: Standard
  Information for Operators }}.
\newblock EMS Tracts in Mathematics, Vol. 18, Eur. Math. Soc. Publ. House,
  Z\"urich, 2012.

\bibitem{RS96B}
T.~Runst and W.~Sickel.
\newblock {\em {Sobolev Spaces of Fractional Order, Nemytskij Operators, and
  Nonlinear PartialDifferential Equations}}.
\newblock De Gruyter Series in Nonlinear Analysis and Applications 3. Walter de
  Gruyterand Co., Berlin, 1996.

\bibitem{ST87B}
H.~Schmeisser and H.~Triebel.
\newblock {\em {Topics in Fourier Analysis and Function Spaces}}.
\newblock Chichester; New York : Wiley, 1987.

\bibitem{SST08}
C.~Schwab, E.~{S\"uli}, and R.~Todor.
\newblock {Sparse finite element approximation of high-dimensional
  transport-dominated diffusion problems}.
\newblock {\em ESAIM Math. Model. Numer. Anal.}, 42:777--819, 2008.

\bibitem{Tem93B}
V.~N. Temlyakov.
\newblock {\em {Approximation of Periodic Functions}}.
\newblock Computational Mathematics and Analysis Series, Nova Science
  Publishers, Inc., Commack, NY., 1993.

\bibitem{Tr78}
H.~Triebel.
\newblock {\em {Besov-Sobolev-Hardy Spaces}}.
\newblock Teubner-Texte Math 9, Teubner, Leipzig,, 1978.

\bibitem{Tri83B}
H.~Triebel.
\newblock {\em {Theory of function spaces}}.
\newblock Birkh\"auser, Basel, 1983.

\bibitem{Tri06B}
H.~Triebel.
\newblock {\em {Theory of Function Spaces III}}.
\newblock Birkh\"auser Verlag, Basel, 2006.

\bibitem{Tri15B}
H.~Triebel.
\newblock {\em {Hybrid Function Spaces, Heat and Navier-Stokes Equations}}.
\newblock European Mathematical Society, 2015.

\bibitem{Ull06}
T.~Ullrich.
\newblock {Function spaces with dominating mixed smoothness. Characterizations
  by differences}.
\newblock {\em Jenaer Schriften zur Mathematik und Informatik},
  Jena:Math/Inf/05/06, 2006.

\bibitem{Vy06}
J.~Vybiral.
\newblock Function spaces with dominating mixed smoothness.
\newblock {\em Diss. Math.}, 436:1--73, 2006.

\bibitem{Yser10}
H.~Yserentant.
\newblock {\em {Regularity and Approximability of Electronic Wave Functions}}.
\newblock Lecture Notes in Mathematics, Springer, 2010.

\end{thebibliography}

\end{document}